\documentclass[12pt,reqno]{article}

\usepackage[usenames]{color}
\usepackage{amssymb}
\usepackage{graphicx}
\usepackage{amscd}
\usepackage{epstopdf}
\usepackage[table]{xcolor}

\usepackage[colorlinks=true,
linkcolor=webgreen,
filecolor=webbrown,
citecolor=webgreen]{hyperref}

\definecolor{webgreen}{rgb}{0,.5,0}
\definecolor{webbrown}{rgb}{.6,0,0}

\usepackage{color}
\usepackage{fullpage}

\usepackage{graphics,amsmath,amssymb}
\usepackage{amsthm}
\usepackage{amsfonts}
\usepackage{latexsym}

\newcommand{\seqnum}[1]{\href{http://oeis.org/#1}{\underline{#1}}}

\begin{document}

%\begin{center}
%\epsfxsize=4in
%\leavevmode\epsffile{logo129.eps}
%\end{center}

\theoremstyle{plain}
\newtheorem{theorem}{Theorem}
\newtheorem{corollary}[theorem]{Corollary}
\newtheorem{lemma}[theorem]{Lemma}
\newtheorem{proposition}[theorem]{Proposition}

\theoremstyle{definition}
\newtheorem{definition}[theorem]{Definition}
\newtheorem{example}[theorem]{Example}
\newtheorem{conjecture}[theorem]{Conjecture}

\theoremstyle{remark}
\newtheorem{remark}[theorem]{Remark}

\begin{center}
\vskip 1cm{\LARGE\bf Sequences of odd numbers, even numbers\\[5pt] and integer squares: gaps in frequency\\[5pt] distributions of unit's digits of minor totals}
\vskip 1cm
{\large  Vladimir L. Gavrikov\\
Institute of Ecology and Geography \\
Siberian Federal University\\
600041 Krasnoyarsk, pr. Svobodnyi 79\\
Russian Federation\\
\href{mailto:vgavrikov@sfu-kras.ru}{\tt vgavrikov@sfu-kras.ru}\\
}
\end{center}

\vskip .2 in

\begin{abstract}
In the paper, I consider appearance of unit's digits in minor totals of a few integer sequences. The sequences include the sequence of even integers, sequence of odd integers and Faulhaber polynomial at $p = 2$. Application of difference tables allows predicting of which digits can appear as unit's digits in minor totals of the sequences. Absence of some digits (``gaps'' in frequency distributions) depends often on numbering system applied. However, in case of odd numbers' integers the gaps are found under all numbering systems with bases from $3$ to $10$.
\end{abstract}

\section{Introduction}
Infinite natural sequence $1,\ 2,\ 3,\ \dots$ has correspondent minor totals $1,\ 3,\ 6,\ 10,\ \dots$ the values of which are given by the well known formula:
\begin{equation}
S_n = \frac{n(n + 1)}{2},
\label{eq:1}
\end{equation}
where $S_n$ is the minor total of the first $n$ members of the sequence. The minor totals possess some interesting properties. In particular, they coincide with binomial coefficients at quadratic terms of the polynomial decomposition of the binomial theorem $(1 + x)^r$, $x$ and $r$ being a real variable and an integer exponent, respectively. One can see the coincidence directly from the famous Pascal's triangle (Fig. \ref{fig1}) in which the sequences in diagonals represent coefficients at linear, quadratic, cubic etc. terms of the theorem.

\begin{figure}[tb]
\center{\includegraphics[width=0.5\textwidth]{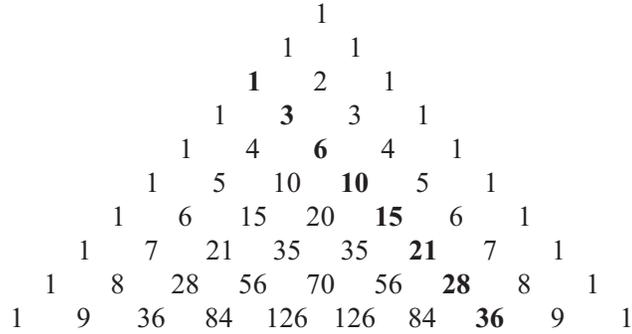}}
\caption{Pascal's triangle. Minor totals of natural sequence (binomial coefficients at the quadratic term) are given in bold face.}
\label{fig1}
\end{figure}

Binomial coefficients have been extensively studied, the focus being often on divisibility of the coefficients by primes \cite{Chen, Guo2013, Guo2014, Pomerance, Winberg}. Recently Gavrikov \cite{Gavrikov} showed the binomial coefficients at quadratic terms (that is, minor totals of the natural sequence) to possess some other properties. The frequencies with which the unit's digits appear in the minor totals may be in part predicted---some digits never appear as unit's digits in the minor totals, i.e., they have zero frequencies. For example, in base-ten numbering system, the minor totals never end with $2,\ 4,\ 7,\ 9$ as unit's digits, which can be strictly proven. On the other hand, in base-eight system, every digit from $0$ to $7$ may appear as unit's digit in the values of the minor totals.

In terms of modular arithmetic, the last propositions may be presented as $S_n \not\equiv \mathsf{A} \pmod {10}, \mathsf{A} \in \{2,\ 4,\ 7,\ 9 \}$ and $S_n \equiv \mathsf{B} \pmod 8, \mathsf{B} \in \{0,\ 1,\ 2,\ 3,\ 4,\ 5,\ 6,\ 7 \}$. Gavrikov \cite{Gavrikov} studied the presence or absence of zero frequencies (gaps) in frequency distribution of unit's digits of minor totals for numbering systems with bases $7$ and $8$, with an approach of ``difference tables'' being suggested. The approach theoretically allows one to predict the gaps in unit's digits for any numbering system.

While the natural sequence is presumably the most general nonnegative integer sequence other integer sub-sequences present interest within the topic. In this study, I consider a number of nonnegative integer sequences with the aim to find out whether the distributions of unit's digits in their minor totals contain gaps. The list of sequences includes sequence of even numbers, sequence of odd numbers, sequences of integer powers of natural numbers.

The basic equation of the analysis looks like
\begin{equation}
S_{Lk+i} = L\cdot m + j,
\label{eq:2}
\end{equation}
where the term on the left is a notation of an integer sequence minor total in base-$L$ numbering system. The term on the right is a representation of the value of the minor total in the same base-$L$ system. Letters $i, j$ denote unit's digits, obviously $0 \leq  i, j \leq (L-1)$. Values of $k$ and $m$ are nonnegative integers.

In the following, to avoid a confusion the letter $S$ in Eq.~(\ref{eq:2}) is substituted by other letters ($T, V, W, \dots$) to denote other integer sequences.

The task of the analysis is to find out if the solution of Eq.~(\ref{eq:2}) can be found in nonnegative integers. That is, if $L$ is given and $k$ and $m$ remain nonnegative integers. Those $j$ values for which the solutions in Eq.~(\ref{eq:2}) is found are possible unit's digits in the minor totals of a given integer sequence.

For minor totals of the natural sequence (Eq.~\ref{eq:1}), it can be empirically shown that the gaps appear in numbering systems with bases $5$, $6$, $7$, $9$, $10$ while there are no gaps in systems with bases $4$, $8$, $16$. Consequently, in the following analyses the ten-base numbering system is taken as a representative of system where the gaps were observed and eight-base numbering system as a representative of systems where the gaps may be not observed.

\section{Sequence of even integers}
Getting of an expression for minor totals for the sequence of even integers $2,\ 4,\ 6,\ 8, \dots$ is quite easy. It is enough to take $2$ out of the brackets, which gives an expression $2(1 + 2 + 3 + \dots )$, i.e. the minor total is double of the minor total for the natural sequence (Eq.~\ref{eq:1}). Thus,
\begin{equation}
T_n = n(n + 1),
\label{eq:3}
\end{equation}
$T$ being the minor total of the sequence of first $n$ even integers.

Obviously, sums of even integers can have in unit's digits only those from $\{ 0,\ 2,\ 4,\ 6,\ 8\}$. It can be observed that some base-ten digits do not empirically appear as unit's digits in minor totals of the sequence of even integers (Fig. \ref{fig2}a). More precisely, the minor totals may not to end with $4$ or $8$. On the other hand, eight-based minor totals contain any base-eight even digits as unit's digits (fig. \ref{fig2}b).

\begin{figure}[tb]
\begin{minipage}[h]{0.49\textwidth}
\center{\includegraphics[width=0.8\textwidth]{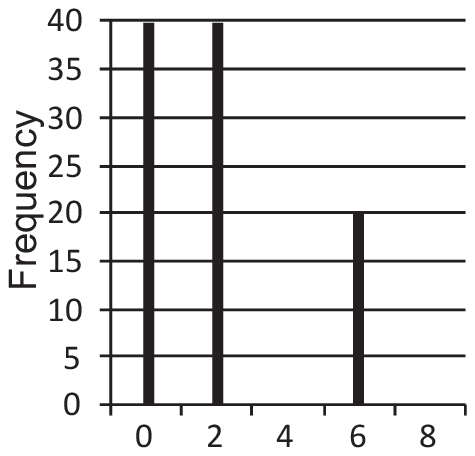} \\ a}
\end{minipage}
\hfill
\begin{minipage}[h]{0.49\textwidth}
\center{\includegraphics[width=0.8\textwidth]{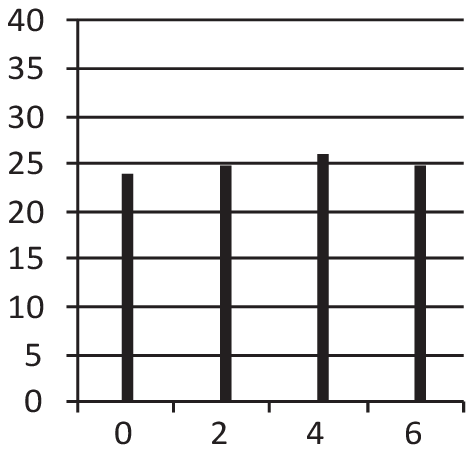} \\ b}
\end{minipage}
\caption{Empirical appearance of unit's digits in minor totals of the sequence of even integers. \textbf{a} -- base-ten numbering system; \textbf{b} -- base-eight system. The frequencies are computed for first hundred minor totals.}
\label{fig2}
\end{figure}

To get conditions determining the appearance of the unit's digits, let us analyze the basic equation in the form
\begin{equation}
T_{Lk+i} = L\cdot m + j.
\label{eq:4}
\end{equation}

Taking into account Eq.~(\ref{eq:3}) gives
\begin{equation}
T_{Lk+i} = L^2k^2 + L\cdot k(2i + 1) + i(i + 1) = L\cdot m + j.
\label{eq:5}
\end{equation}

For further analysis, $m$ may be expressed from Eq.~(\ref{eq:5}) as
\begin{equation}
m = L\cdot k^2 + k(2i + 1) + \frac{i(i + 1) - j}{L}.
\label{eq:6}
\end{equation}

Provided $i,\ j,\ k,\ L$ all are integers, the question is whether $m$ is also a nonnegative integer. Obviously, $L\cdot k^2 + k(2i + 1)$ is always an integer while $(i(i + 1) - j)/L$ may be fractional.

Because $i$ and $j$ vary freely between $0$ and $L - 1$ the expression $(i(i + 1) - j)$ is a sort of two-dimensional table containing all the possible differences between $i(i + 1)$ and $j$, a \textit{difference table} that has dimensions $L\times L$.

Thus, considering all the differences and finding those that are divisible by $L$ without a remainder helps to find the $j$ values that ensure that the Eq.~(\ref{eq:4}) can be solved in nonnegative integers. Those $j$s are therefore the unit's digits that can appear in minor totals of the sequence of even numbers.

\begin{proposition}
In base-ten numbering system, among the minor totals of the sequence of even numbers there are no such than have $4$ and $8$ as unit's digits, i.e., $T_n \not\equiv 4 \pmod {10}$ and $T_n \not\equiv 8 \pmod {10}$.
\end{proposition}

\begin{proof}
A $10\times 10$ difference table should be considered. Because $j$ cannot be odd it is enough to consider only even $j$ values.

\begin{table}[htb]
  \centering
    \begin{tabular}{rrrrrrrrrrrr}
    \hline
    $j$&&\multicolumn{10}{c}{$i(i + 1)$}\\ \cline{3-12}
     &    & 0     & 2     & 6     & 12     & 20    & 30    & 42 & 56 & 72 & 90 \\
    \hline
    \textbf{0} &   & \cellcolor{lightgray} 0    & 2     & 6     & 12     & \cellcolor{lightgray}20    &\cellcolor{lightgray} 30    & 42 & 56 & 72 &\cellcolor{lightgray} 90 \\
    \textbf{2} &   & -2    & \cellcolor{lightgray} 0     & 4     &\cellcolor{lightgray}10     & 18    &  28   & \cellcolor{lightgray} 40 & 54 &\cellcolor{lightgray} 70 & 88 \\
    4 &   & -4    & -2    & 2     & 8     & 16    & 26    & 38 & 52 & 68 & 86 \\
    \textbf{6} &   & -6    & -4    & \cellcolor{lightgray} 0  & 6     & 14     & 24    & 36 & \cellcolor{lightgray}50 & 66 & 84 \\
    8 &   & -8    & -6    & -2    & 4     & 12     & 22    & 34 & 48 & 64 & 82 \\
    \hline
    \end{tabular}
\caption{A $10\times 10$ difference table. Values of differences divisible by $10$ without a remainder are given on gray background. Values of $j$ satisfying the condition ``$m$ is a nonnegative integer'' are given in bold face.}    
  \label{tab:1}
\end{table}

Table \ref{tab:1} shows that $j$s satisfying the divisibility of the differences by $10$ without a remainder are $0$, $2$ and $6$. No such differences are found for $j$ equal to $4$ or $8$. Thus, the solution of Eq.~(\ref{eq:4}) cannot contain $4$ and $8$ and $T_n \not\equiv 4 \pmod {10}$ and $T_n \not\equiv 8 \pmod {10}$.
\end{proof}

\begin{proposition}
In base-eight numbering system, all digits of the system appear as unit's digits of the minor totals, i.e., $T_n \equiv \mathsf{C} \pmod {10}, \mathsf{C} \in \{0,\ 2,\ 4,\ 6\}$.
\end{proposition}
\begin{proof}
A $8\times 8$ difference table should be considered. Table \ref{tab:2} shows that for every $j \in \{0,\ 2,\ 4,\ 6\}$ at least one difference value can be found that is divisible by $8$ without a remainder. It means that all digits of the base-eight numbering system may appear as unit's digits in base-eight minor totals of the sequence of even integers.

\begin{table}[htb]
  \centering
    \begin{tabular}{rrrrrrrrrrrr}
    \hline
    $j$&&\multicolumn{8}{c}{$i(i + 1)$}\\ \cline{3-10}
     &    & 0     & 2     & 6     & 12     & 20    & 30    & 42 & 56\\
    \hline
    \textbf{0} &   & \cellcolor{lightgray} 0    & 2     & 6     & 12     & 20    &30    & 42 &\cellcolor{lightgray}  56 \\
    \textbf{2} &   & -2    & \cellcolor{lightgray} 0     & 4     &10     & 18    &  28   & \cellcolor{lightgray} 40 & 54 \\
    \textbf{4} &   & -4    & -2    & 2     &\cellcolor{lightgray} 8     &\cellcolor{lightgray} 16    & 26    & 38 & 52 \\
    \textbf{6} &   & -6    & -4    & \cellcolor{lightgray} 0  & 6     & 14     &\cellcolor{lightgray} 24    & 36 & 50 \\
    \hline
    \end{tabular}
  \caption{A $8\times 8$ difference table. Values of differences divisible by $8$ without a remainder are given on gray background. Values if $j$ satisfying the condition ``$m$ is a nonnegative integer'' are given in bold face.}    
  \label{tab:2}
\end{table}
\end{proof}

\begin{remark}
Interestingly, Table \ref{tab:1} predicts that \textit{odd} digits do can appear as unit's digits in minor totals of the sequence of \textit{even} integers, at least in some peculiar cases. For example, let us consider sequence of even integers in base-five numbering system. For the system, the row of the difference table for $j = 1$ contains $-1,\ 1,\ 5,\ 11,\ 19$, which means that unity should appear as unit's digit in the minor totals (because $5$ is divided by $L = 5$ without a remainder). In fact, $2_5 + 4_5 = 11_5$ (because $6_{10} = 11_5$).
\end{remark}

\section{Sequence of odd integers}
The sequence of odd integers $1 + 3 + 5 + 7 +\dots$ has a well known expression for the minor totals of the first $n$ members as
\begin{equation}
U_n = n^2,
\label{eq:7}
\end{equation}
$U$ being the minor total of the sequence.

Empirical observations show that some digits do not appear as unit's digits in minor totals of the sequence not only for base-ten numbering system but for all bases from $3$ to $10$. Figure \ref{fig3} demonstrates that gaps in the empirical frequency distributions appear both in base-ten and base-eight systems.

\begin{figure}[tb]
\begin{minipage}[h]{0.49\textwidth}
\center{\includegraphics[width=0.8\textwidth]{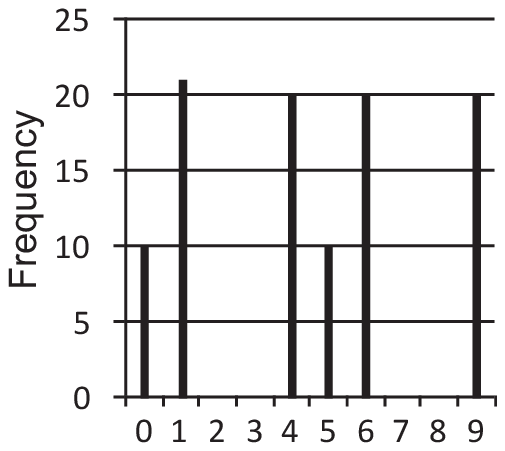} \\ a}
\end{minipage}
\hfill
\begin{minipage}[h]{0.49\textwidth}
\center{\includegraphics[width=0.8\textwidth]{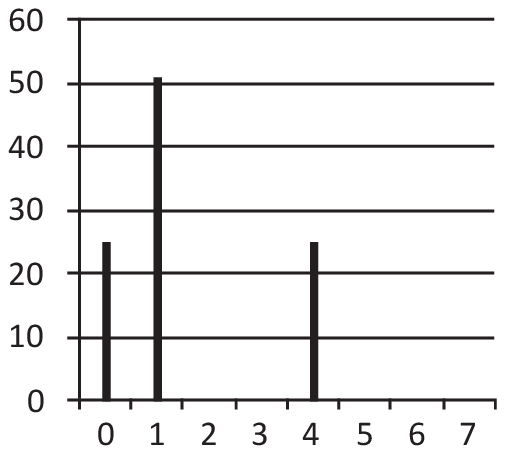} \\ b}
\end{minipage}
\caption{Empirical appearance of unit's digits in minor totals of the sequence of odd numbers. \textbf{a} -- base-ten numbering system; \textbf{b} -- base-eight system. The frequencies are computed for first hundred minor totals.}
\label{fig3}
\end{figure}

\begin{proposition}
In all numbering systems from base-three to base-ten, there are gaps in frequency distributions (zero frequencies) of unit's digits of minor totals of the sequence of odd integers.
\end{proposition}
\begin{proof}
It is necessary to consider the difference tables from $3 \times 3$ to $10 \times 10$ and to check the differences for the divisibility by $3$ to $10$, correspondingly. Technically, the tables are nested one into each other, the table $10 \times 10$ containing all others (Table \ref{tab:3}).

The inspection of the tables shows that:
\begin{itemize}
\item[-] $U_n \not\equiv \mathsf{D_{10}} \pmod {10}, \mathsf{D_{10}} \in \{2,\ 3,\ 7,\ 8\}$ (there are no differences divisible by $10$ without a remainder in rows of $j \in \{2,\ 3,\ 7,\ 8\}$)
\item[-] $U_n \not\equiv \mathsf{D_{9}} \pmod {9}, \mathsf{D_{9}} \in \{2,\ 3,\ 5,\ 6,\ 8\}$
\item[-] $U_n \not\equiv \mathsf{D_{8}} \pmod {8}, \mathsf{D_{8}} \in \{2,\ 3,\ 5,\ 6,\ 7\}$
\item[-] $U_n \not\equiv \mathsf{D_{7}} \pmod {7}, \mathsf{D_{7}} \in \{3,\ 5,\ 6\}$
\item[-] $U_n \not\equiv \mathsf{D_{6}} \pmod {6}, \mathsf{D_{6}} \in \{2,\ 5\}$
\item[-] $U_n \not\equiv \mathsf{D_{5}} \pmod {5}, \mathsf{D_{5}} \in \{2,\ 3\}$
\item[-] $U_n \not\equiv \mathsf{D_{4}} \pmod {4}, \mathsf{D_{4}} \in \{2,\ 3\}$
\item[-] $U_n \not\equiv 2 \pmod {3}$
\end{itemize}

Therefore, in minor totals of odd integer sequence under all numbering systems from bases $3$ to $10$ there are values of $j$ that do not appear as unit's digits of the minor totals.

\begin{table}[htb]
  \centering
  
    \begin{tabular}{rrrrrrrrrrrr}
    \hline
    $j$&&\multicolumn{10}{c}{$i^2$}\\ \cline{3-12}
     &    & 0     & 1     & 4     & 9      & 16    & 25    & 36 & 49 & 64 & 81 \\
    \hline
    0 &   &  \textbf{0}    & \multicolumn{1}{r|}{1}     & 4     & 9      & 16    & 25    & 36 & 49 & 64 & 81 \\
    1&    & -1    & \multicolumn{1}{r|}{\textbf{0}}     & 3     & 8      & 15    & 24    & 35 & 48 & 63 & 80 \\ \cline{3-4}
    2 &   & -2    & -1    &\multicolumn{1}{r|}{2}     & 7      & 14    & 23    & 34 & 47 & 62 & 79 \\ \cline{5-5}
    3 &   & -3    & -2    & 1     &\multicolumn{1}{r|}{6}      & 13    & 22    & 33 & 46 & 61 & 78 \\ \cline{6-6}
    4 &   & -4    & -3    & \textbf{0}     & 5      &\multicolumn{1}{r|}{12}    & 21    & 32 & 45 & 60 & 77 \\ \cline{7-7}
    5 &   & -5    & -4    &-1     & 4      & 11    &\multicolumn{1}{r|}{20}    & 31 & 44 & 59 & 76 \\ \cline{8-8}
    6 &   & -6    & -5    &-2     & 3      & 10    & 19    &\multicolumn{1}{r|}{30} & 43 & 58 & 75 \\ \cline{9-9}
    7 &   & -7    & -6    &-3     & 2      &  9    & 18    & 29 &\multicolumn{1}{r|}{42} & 57 & 74 \\ \cline{10-10}
    8 &   & -8    & -7    &-4     & 1      &  8    & 17    & 28 & 41 &\multicolumn{1}{r|}{56} & 73 \\ \cline{11-11}
    9 &   & -9    & -8    &-5     & \textbf{0}      &  7    & 16    & 27 & 40 & 55 & 72 \\
    \hline
    \end{tabular}
    \caption{A $10\times 10$ difference table. Short vertical and horizontal lines delimit all other difference tables (bases from $2$ to $9$) nested in to the $10\times 10$ table. Zero values of the differences are highlighted in bold face.}
  \label{tab:3}
\end{table}
\end{proof}

\begin{remark}
Obviously, zero values of the differences in tables ensure that the correspondent $j$ value appear as unit's digits in minor totals of an integer sequence. As for the odd integer sequence, it may hypothesized that all digits appear as unit's digits in the only case of base-two numbering system. So, hypothetically, $\forall L > 2 \ \exists j \ \forall i \ (i^2 - j) \equiv \mathsf{D}, \mathsf{D} \neq 0 \pmod L$, $L$ being the base of numbering system and $j,\ i$ being nonnegative integers, $0 \leq  j,\ i \leq L - 1$.
\end{remark}

\section{A case of Faulhaber polynomials}
Natural sequence $1,\ 2,\ 3,\ 4, \dots$ is a particular case of sequence of powers like  $1^p,\ 2^p,$ $ 3^p,\ 4^p, \dots$, $p$ being a nonnegative integer. Minor totals of the latter is known to be Faulhaber polynomials:
\begin{equation}
1^p + 2^p + 3^p + 4^p, \dots
\label{eq:8}
\end{equation}
For definite $p$ in Exp.~(\ref{eq:8}), formulas for the minor totals are known. Let us consider a case of $p = 2$ with known minor total $V_n$:
\begin{equation}
V_n = 1^2 + 2^2 + 3^2 + 4^2 + \dots + n^2 = \frac{n(n + 1)(2n + 1)}{6}.
\label{eq:9}
\end{equation}
Application of the basic Eq.~(\ref{eq:2}) to Eq.~(\ref{eq:9}) gives
\begin{equation}
V_{Lk + i} = \frac{(Lk + i)(Lk + (i + 1))(2Lk + (2i + 1))}{6} = L\cdot m + j.
\label{eq:10}
\end{equation}
Extracting of $m$ from Eq.~(\ref{eq:10}) gives
\begin{eqnarray}
m&= & \Big( 2L^2k^3 + 3Lk^2(2i+1) + k\bigl( 2(i+1)(2i+1) + \nonumber \\
&+ &i(2i+1) + 2i(i+1)\bigr)\Big) /6 + \label{eq:11}\\
&+ &\frac{\cfrac{i(i + 1)(2i + 1)}{6}- j}{L}. \label{eq:12}
\end{eqnarray}
The term (\ref{eq:11}) may be an integer or it may have remainders from division by $6$. The remainders obviously belong to $\{1/6,\ 2/6,\ 3/6,\ 4/6,\ 5/6\}$. The term (\ref{eq:12}) represents results of divisions of the difference table values $\frac{i(i + 1)(2i + 1)}{6}- j$ by $L$. Because $m$ must be a nonnegative integer the remainders in (\ref{eq:12}) must be either zero or a multiple of $1/6$ complementary to the remainders in term (\ref{eq:11}), i.e., producing the unity after summation.
\begin{proposition}
Under base-ten numbering system, digits $2,\ 3,\ 7,\ 8$ do not appear as unit's digits in the minor totals $V_n$ of Faulhaber polynomial with $p = 2$, i.e., $V_n \not\equiv \mathsf{E} \pmod {10}, \mathsf{E} \in \{ 2,\ 3,\ 7,\ 8\}$.

Under base-eight numbering system, all the digits of the system appear as unit's digits in the minor totals $V_n$ of Faulhaber polynomial with $p = 2$, i.e., $V_n \equiv \mathsf{E} \pmod {8}, \mathsf{E} \in \{0,\ 1,\ 2,\ 3,\ 4,\ 5,\ 6,\ 7\}$.
\end{proposition}

\begin{proof}
It is necessary to consider a $10\times 10$ difference table and the correspondent $8\times 8$ table. Table \ref{tab:4} gives the idea of remainders from divisions of $i(i + 1)(2i + 1)/6 - j$ by $L$. Those $j$-rows are sought in which the differences are located that are divisible by $10$ either with zero remainder or with $3/6$ (that is, $1/2$) remainder. These $j$-rows are $0,\ 1,\ 4,\ 5,\ 6,\ 9$ (the differences are given on gray background). Consequently, $j$ cannot take values $2,\ 3,\ 7,\ 8$ without violating of the condition ``$m$ is a non-negative integer''. Thus, under base-ten system, integers $2,\ 3,\ 7,\ 8$ do not appear as unit's digits in minor totals of Faulhaber polynomial with $p = 2$.

On the other hand, in all $j$-rows of the $8\times 8$ table, there are differences (given in bold face in Table \ref{tab:4}) that while divided by $8$ give remainders of either $0$ or $3/6$. Therefore, all digits of base-eighth system appear as unit's digits in minor totals of Faulhaber polynomial with $p = 2$.

\begin{table}[htb]
  \centering
  
    \begin{tabular}{rrrrrrrrrrrrr}
    \hline
    $j$&&\multicolumn{10}{c}{$i(i + 1)(2i + 1)/6$}\\ \cline{3-13}
     &    & 0     & 1     & 5     & 14      & 30    & 55    & 91 & 140 & & 204 & 285 \\
    \hline
    0 &  &\cellcolor{lightgray} \textbf{0} & 1 &\cellcolor{lightgray} 5 & 14 &\cellcolor{lightgray} 30 &\cellcolor{lightgray} 55 & 91 &\cellcolor{lightgray} \textbf{140}&\multicolumn{1}{r|}{} & 204 &\cellcolor{lightgray} 285 \\
    1 &  & -1 &\cellcolor{lightgray}\textbf{0} & \textbf{4} & 13 & 29 & 54 &\cellcolor{lightgray}90 & 139 &\multicolumn{1}{r|}{} & 203 & 284 \\
    2 &  & -2 & -1 & 3 &\textbf{12} &\textbf{28} & 53 & 89 & 138 &\multicolumn{1}{r|}{} & 202 & 283 \\
    3 &  & -3 & -2 & 2 & 11 & 27 &\textbf{52} &\textbf{88} & 137 &\multicolumn{1}{r|}{} & 201 & 282 \\
    4 &  & -4 & -3 & 1 &\cellcolor{lightgray} 10 & 26 & 51 & 87 & 136 &\multicolumn{1}{r|}{} &\cellcolor{lightgray} 200 & 281 \\
    5 &  & -5 & -4 &\cellcolor{lightgray}\textbf{0} & 9 &\cellcolor{lightgray} 25 &\cellcolor{lightgray} 50 & 86 &\cellcolor{lightgray} 135&\multicolumn{1}{r|}{} & 199 &\cellcolor{lightgray} 280 \\
    6 &  & -6 & -5 & -1 &\textbf{8} &\textbf{24} & 49 &\cellcolor{lightgray} 85 & 134 &\multicolumn{1}{r|}{} & 198 & 279 \\
    7 &  & -7 & -6 & -2 & 7 & 23 &\textbf{48} &\textbf{84} & 133 &\multicolumn{1}{r|}{} & 197 & 278 \\ \cline{3-11}
    8 &  & -8 & -7 & -3 & 6 & 22 & 47 & 83 & 132 & & 196 & 277 \\
    9 &  & -9 & -8 & -4 &\cellcolor{lightgray} 5 & 21 & 46 & 82 & 131 & &\cellcolor{lightgray} 195 & 276 \\
    \hline
    \end{tabular}
    \caption{A $10\times 10$ difference table. Vertical and horizontal lines delimit a $8\times 8$ difference table nested in the $10\times 10$ table. Values divisible by $10$ with the remainders $0$ and $3/6$ are given on gray background. Values divisible by $8$ with the remainders $0$ and $3/6$ are highlighted in bold face.}
  \label{tab:4}
\end{table}
\end{proof}

\begin{remark}
Empirical observations (Fig. \ref{fig4}) support the inferences.

Obviously, base-six numbering system is a special case in the context of Eq.~(\ref{eq:11}-\ref{eq:12}) because any division by $6$ will produce remainders of either $0$ or multiples of $1/6$. Thus, all the digits from $0$ to $5$ will appear as unit's digits in the minor totals of Faulhaber polynomial with $p = 2$ under base-six system.

Also, it may be noted that the frequency distributions of unit's digits in squared integers (Fig. \ref{fig3}) and in the sums of squared integers (Fig. \ref{fig4}) have the same gaps -- $\{ 2,\ 3,\ 7,\ 8\}$.
\end{remark}

\begin{figure}[tb]
\begin{minipage}[h]{0.49\textwidth}
\center{\includegraphics[width=0.8\textwidth]{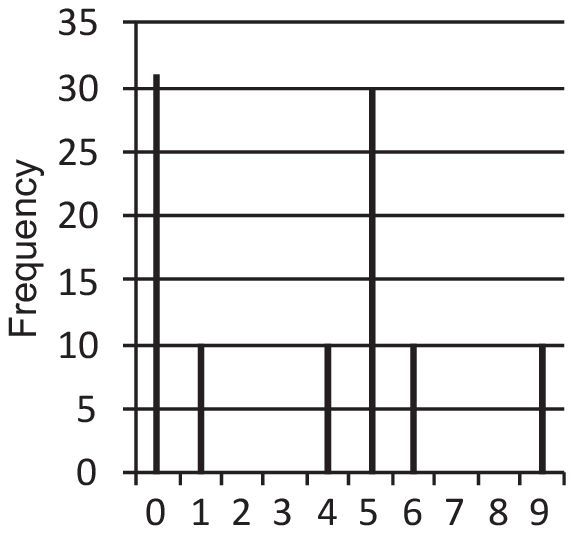} \\ a}
\end{minipage}
\hfill
\begin{minipage}[h]{0.49\textwidth}
\center{\includegraphics[width=0.8\textwidth]{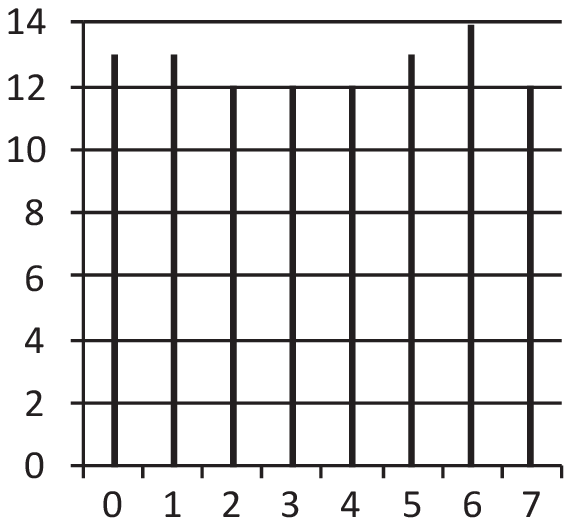} \\ b}
\end{minipage}
\caption{Empirical appearance of unit's digits in Faulhaber polynomials. \textbf{a} -- base-ten numbering system; \textbf{b} -- base-eight system. The frequencies are computed for first hundred minor totals.}
\label{fig4}
\end{figure}

\section{Conclusion}
The approach of difference tables allows one to make predictions regarding appearance of unit's digits of minor totals of a number of integer sequences. In this study, I considered a question as to whether all digits of a numbering system can appear as unit's digits in minor totals of i) the sequence of even integers, ii) sequence of odd integers, and in iii) Faulhaber polynomials with $p = 2$.

Many integer sequences \cite{Gavrikov} have gappy frequency distributions of unit's digits in their minor totals. This means that some digits never appear as unit's digits in the minor totals. This property varies however dependently on numbering system applied and the difference tables can show why the gaps happen. Particularly, it has been shown that minor totals of odd numbers (i.e., squared integers) have gaps in frequencies of unit's digits under all the numbering systems from bases $3$ to $10$.

A limitation is that the approach works only if a formula for minor totals is known. For example, no minor total formula is known for the sequence of primes. It is empirically easy to show that all the digits appear as unit's digits in the minor totals of primes independently of numbering system base (from 2 to 10) but no analytical method can be applied to prove it.

\newpage
\bibliographystyle{jis}
\bibliography{mybibfile}

\bigskip
\hrule
\bigskip

\noindent 2010 {\it Mathematics Subject Classification}:
Primary 11Axx; Secondary 62P10.

\noindent \emph{Keywords: }
Sequence of even integers, sequence of odd integers, Faulhaber polynomials, minor total, unit's digits.

\bigskip
\hrule
\bigskip

\noindent (Concerned with sequences
\seqnum{A005843}, \seqnum{A005408}, \seqnum{A000290}.)

\end{document}